\documentclass[11pt]{article}
\usepackage[T1]{fontenc}
\usepackage{lmodern}
\usepackage{float}
\usepackage{amsfonts,amsmath,amssymb,graphicx,color}
\usepackage{amsthm}
\usepackage{algorithm}
\usepackage{algpseudocode}
\usepackage{psfrag}
\usepackage{pdflscape}
\usepackage{multirow} 
\usepackage{enumitem}
\usepackage{framed}
\usepackage{sidecap}
\usepackage{tikz}
\usepackage{subcaption}
\usetikzlibrary{arrows}
\usetikzlibrary{patterns,arrows,decorations.pathreplacing}
\graphicspath{ {./figures/} }
    \usepackage{soul}
\usepackage[normalem]{ulem}
\usepackage{url}
\usepackage[numbers, sort&compress ]{natbib} 
\usepackage{microtype}

\numberwithin{figure}{section}
\numberwithin{table}{section}

\newcommand{\R}{\mathbb{R}}
\newcommand{\Rn}{\R^n}

\newtheorem{thm}{Theorem}[section]
\newtheorem{Def}[thm]{Definition}

\newtheorem{lem}[thm]{Lemma}
\newtheorem{cor}[thm]{Corollary}
\newtheorem{assum}[thm]{Assumptions}
\numberwithin{equation}{section}

\newcommand{\comment}[1]{}

\newcommand{\be}{\begin{equation}}
\newcommand{\ee}{\end{equation}}

\newcommand{\bea}{\begin{eqnarray}}
\newcommand{\eea}{\end{eqnarray}}
\newcommand{\beqa}{\begin{eqnarray}}
\newcommand{\eeqa}{\end{eqnarray}}
\newcommand{\beann}{\begin{eqnarray*}}
\newcommand{\eeann}{\end{eqnarray*}}

\newcommand{\bmat}{\left[ \begin{array}}
\newcommand{\emat}{\end{array} \right]}

\newcommand{\beq}{\begin{equation}}
\newcommand{\eeq}{\end{equation}}

\newcommand{\bproof}{\begin{description} \item[{\it Proof}.] ~ }
\newcommand{\eproof}{\hspace*{\fill}$\Box$\medskip \end{description}}

\newcounter{algo}[section]

\newcounter{prog}[section]

\title{Constrained Optimization in the Presence of Noise}
\author{       
        Figen Oztoprak\thanks{Artelys Corporation}
       \and     
               Richard Byrd \thanks{Computer Science Department, University of Colorado, Boulder, USA } 
        \and 
       Jorge Nocedal \thanks{Department of Industrial Engineering and Management Sciences, Northwestern University,  USA.  This author was supported by National Science Foundation grant  DMS-2011494, AFOSR grant FA95502110084, and  ONR grant N00014-21-1-2675.} 
      }

\date{\today}
\begin{document}

\maketitle

\begin{abstract}
The problem of interest is the minimization of a nonlinear function subject to nonlinear equality constraints using a sequential quadratic programming (SQP) method. The minimization must be performed while observing only noisy evaluations of the objective and constraint functions. In order to obtain stability, the classical SQP method is modified by relaxing the standard Armijo line search based on the noise level in the functions, which is assumed to be known.
Convergence theory is presented giving conditions under which the iterates converge to a neighborhood of the solution characterized by the noise level and the problem conditioning. The analysis assumes that the SQP algorithm does not require regularization or trust regions. Numerical experiments indicate that the relaxed line search improves the practical performance of the method on problems involving uniformly distributed noise.  One important application of this work is in the field of derivative-free optimization, when  finite differences are employed to estimate gradients.
\end{abstract}


\bigskip
\section{Introduction}
\label{intro}
\setcounter{equation}{0}

Let us consider the equality constrained nonlinear optimization problem
\begin{equation} \label{problem}
 \min_x f(x) \qquad\mbox{s.t.} \ \ c(x) = 0,
\end{equation}
where $f : \Rn \rightarrow \R$ and $c(x) : \Rn \rightarrow \R^m$ are smooth functions. We assume that the minimization must be performed while observing approximate evaluations $\tilde f(x), \tilde c(x)$ of the functions $f,c$ and their derivatives.

 We consider the application of a sequential quadratic programming (SQP) algorithm that employs an $\ell_1$ merit function to control the stepsize. The goal of the paper is to study the effect of noise on the behavior of the SQP algorithm, particularly the achievable accuracy in the solution, and to highlight the aspects of the algorithm that are most susceptible to errors (or noise)---and redesign them. This work was motivated by applications in which the derivatives of $f$ and $c$ are approximated by finite differences \cite{shi2021numerical}, and thus contain errors, but the algorithm and analysis apply to the more general setting when stochastic or deterministic noise are present in both the function and derivative evaluations.

Let us define 
\begin{equation} \label{grads}
    g(x) = \nabla f(x), \qquad J(x) = \nabla c(x) \in \mathbb{R}^{m \times n}, \quad m<n,
 \end{equation}
  and let $\tilde g(x), \tilde J(x)$  be the corresponding noisy evaluations.  
The iteration of the SQP algorithm is given by
\begin{align} \label{iteration}
& x_{k+1} = x_k +\alpha_k d_k ,
\end{align}
where $d_k$ is the solution of the quadratic subproblem 
\begin{align}   
\min_{d \in \mathbb{R}^n} & \  \tfrac{1}{2} d^T H_k d + \tilde g_k^T d \label{kkt} \\
{\rm s.t.} &  \ \tilde c_k + \tilde J_k d=0 , \label{lincons}
\end{align}
and the steplength $\alpha_k >0$ is chosen so as to ensure sufficient decrease in the merit function 
\begin{equation}  \label{merit}
   \tilde \phi(x) = \tilde f(x) + \pi \| \tilde c(x) \|_1 
\end{equation}
 when the iterates are far away from a solution. Here $\pi >0$ is a penalty parameter that is adjusted during the course of the optimization.
The symmetric matrix $H_k$ is generally chosen as an approximation to the Hessian of the Lagrangian. However, in this paper we  assume that $H_k$ is a  multiple of the identity matrix, 
\begin{equation} \label{hdef}
   H_k= \beta_k I \qquad \beta_k >0,
\end{equation}
because allowing more general choices  introduces more constants in the analysis  without contributing to the main goals of this investigation.

As in the noiseless case, the control of the penalty parameter in \eqref{merit} is of critical importance in the SQP algorithm. $\pi$ should be chosen so that the SQP direction $d_k$ is a descent direction for $ \tilde \phi$ at $x_k$, and it should provide adequate control on the size of $\alpha_k$. The proposed algorithm has the  general form of a classical SQP method \cite{Flet87}, specialized to the case when $H_k$ is a multiple of the identity matrix, and introduces a modification in the line search designed to handle noise. 

We assume throughout that the noise in the function and gradient evaluations is bounded by some constants $\epsilon_f$ and $ \epsilon_c$. This is not always the case in practice (e.g. when noise is Gaussian) but it covers many important practical settings, including computational noise \cite{more2012estimating}. Furthermore, we assume that $\epsilon_f,  \epsilon_c$  are  known, or can be estimated, and that the algorithm has access to them. 

This study was motivated by some practical computations performed by the authors using the {\sc knitro} software package \cite{ByrdNoceWalt06}. They selected a few challenging nonlinear optimization problems involving equality and inequality constraints, injected noise in the objective and constraints, and computed derivatives using noise-aware finite difference formula; see e.g.,  Mor\'e and Wild \cite{more2012estimating} and Berahas et al. \cite{berahas2019derivative}. They observed that, for low levels of noise,  {\sc knitro}  returned acceptable answers, even though one might suspect the default algorithm to be brittle in this setting. As the noise level was increased, the quality of the solution deteriorated markedly, suggesting that classical optimization methods should be redesigned to handle noise. To guide this investigation, it is essential to develop a convergence theory. In this paper, we focus on the case when noise cannot be diminished, and  characterize the accuracy of a noise tolerant optimization algorithm.

 As a first step in this investigation, we find it convenient to consider equality constrained optimization, and study the performance of a sequential quadratic optimization method, which is a simple method in this setting and must yet confront some important challenges raised by the presence of noise.

\subsection{Contributions of this work}
The main contribution of this paper is the development of a convergence theory for a classical sequential quadratic programming (SQP) algorithm for equality constrained optimization in the presence of noise. It is shown that, by introducing a relaxation in the line search procedure while keeping all other components of the SQP method unchanged,  the iterates of the algorithm reach an acceptable neighborhood $C_1$ of the solution defined by a stationarity measure for the problem. Furthermore, once the iterates enter $C_1$ they cannot escape a larger neighborhood $C_2$ and must revisit $C_1$ an infinite number of times. The analysis gives a detailed characterization of these neighborhoods in terms of the noise level and problem characteristics. Numerical experiments show that the relaxed line search is, in fact, beneficial in practice.

Our convergence results assume that errors in function and gradients are bounded, and the analysis is deterministic, yielding somewhat pessimistic bounds. We believe, however, that the results can be useful in the design of robust constrained optimization methods. Specifically, our analysis suggests that only slight modifications are needed so that a classical SQP method is able to handle bounded noise. 

\subsection{Literature Review }

Early work on constrained optimization in the presence of noise is reviewed by Poljak (a.k.a. Polyak) \cite{poljak1978nonlinear}. His study includes penalty, Lagrange, or extended Lagrange functions, and establishes 
 probabilistic convergence theorems provided the steplength is chosen small enough from the start. Hintermueller \cite{hintermuller2002solving} studies a penalty SQP method in which equality constraints are replaced with upper and lower bounding surrogates.
Assuming that the noise level in the function is known, it is shown that in the limit the bounds contain a solution.  Schittkowski \cite{schittkowski2014nlpqlp}  uses a non-monotone line search to handle errors due to approximate function and derivative evaluations. His algorithm was implemented in the NLPQLP software, which is reported to be successful in practice, but no convergence theory were presented. 


The work that is most closely related to this study is \cite{berahas2021stochastic,berahas2021sequential,curtis2021inexact}. In \cite{berahas2021sequential}, an SQP method for equality constrained optimization is presented to handle the case when the objective function is stochastic and the constraints are deterministic. The stepsize is obtained by adaptively estimating Lipschitz constants in place of a line search. Conditions for convergence in expectation are established. \cite{berahas2021stochastic} considers the case when Jacobians can be rank deficient, proposes a step decomposition approach, and presents compelling numerical results. \cite{curtis2021inexact} studies an SQP algorithm with an inexact step computation for the same problem setting. These three papers give careful attention to the behavior of the penalty parameter. For example,  in \cite{berahas2021sequential} the penalty parameter is chosen in a way that provides sufficient descent in the quadratic model of the merit function in the deterministic setting.  In the stochastic setting, they employ the stochastic gradient of the objective in the same formulae for updating the penalty parameter, but they can no longer guarantee that the resulting penalty parameter will be large enough and bounded.  They prove their convergence results assuming that the penalty parameter is well behaved.  Then, they discuss the probability of having small penalty values, and note that the boundedness issue is resolved  by making the same assumption as in this paper, namely that noise is always bounded. 


\medskip \noindent
{\em Notation}. We let $\| \cdot \|$ denote the $\ell_2$ norm, unless otherwise stated. As is the convention, $f_k$ stands for $f(x_k)$ and similarly for other functions. The terms {\em error} and {\em noise} in the functions is used interchangeably. Since we assume absolute bounds on these quantities, the distinction between them is not important in this study.

\section{The Algorithm}
Before presenting the algorithm, we introduce some notation. We model the first-order change in the merit  function $\phi$ at an iterate $x_k$ as
\begin{equation} \label{lineart}
\tilde \ell(x_k;d_k)=\tilde g_k^Td_k +\pi_k\|\tilde c_k+ \tilde J_k d_k\|_1 - \pi_k \|\tilde c_k \|_1 .
\end{equation}
We also define
\begin{equation} \label{linden}
  \hat  \lambda_k =  (\tilde J_k \tilde J_k^T)^{-1} \tilde J_k  \tilde g_k,
\end{equation}
which is the standard least squares multiplier estimate \cite[eqn(18.21)]{mybook}, accounting for noisy function evaluations. We assume that $ \tilde J_k$  is full rank for all $k$, hence $\hat \lambda_k$ is well defined.
 
 The penalty parameter will be updated using the following classical formula \cite[eqn(18.32)]{mybook}. Given a (fixed) parameter $\tau \in (0,1)$, we set at every iteration 
\begin{equation}  \label{updatepi}
    \pi_{k}=   \left\{ 
    \begin{array}{rl}
    \pi_{k-1} \ & \mbox{if} \ \ {\pi_{k-1}} \geq  \frac{1}{1-\tau} \|\hat \lambda_k\|_\infty \\
      \frac{2}{1-\tau} \| \hat \lambda_k\|_\infty \ & \mbox{otherwise.} 
     \end{array} 
     \right.
      \end{equation}
      The factor 2 in the second line of \eqref{updatepi} is introduced so that when $\pi_k$ is increased, it is increased substantially. We will see that this rule ensures that $\pi_k$ is eventually fixed.
(In general, SQP methods do not set  $H_k  = \beta_k I$. In that case, using the least squares multiplier estimate in \eqref{updatepi} will not lead to a convergent method.)

The algorithm for solving problem \eqref{problem}, when only noisy evaluations of the functions $\tilde f, \tilde c, \tilde g, \tilde J$ are available, is  as follows. 
\begin{algorithm}[htp] 
	\caption{Noise Tolerant SQP Algorithm}
	\label{algorithm}
	{\bf Input:} Initial iterate $x_0$,  initial merit parameter $\pi_{-1}>0$,   bounds $\epsilon_f, \epsilon_c$ on the noise \eqref{noise1}, and constants $ \tau, \nu  \in (0,1)$.\\
	Set $k \leftarrow 0$\\
	{\bf Repeat} until a termination test is satisfied:
	\begin{algorithmic}[1]
	        \State Compute $\beta_k >0$ and set $H_k = \beta_k I$ in \eqref{kkt}
		\State Compute $d_k$ by solving \eqref{kkt}-\eqref{lincons}
		\State Compute $\hat \lambda_k$ via \eqref{linden}
		\State Update  penalty parameter $\pi_k$ by \eqref{updatepi}
		\State Compute $\tilde \ell(x_k; d_k)$ as in \eqref{lineart} 
		\State Set $\epsilon_R= 2(\epsilon_f +\pi_k \epsilon_c)$
		\State Choose steplength $\alpha_k > 0$ such that
		\begin{equation} \label{relaxed}
\tilde \phi(x_k+\alpha_k  d_k) \leq \tilde \phi(x_k) +\nu \alpha_k \tilde \ell(x_k; d_k) + \epsilon_R,
\end{equation}
		\State Compute new iterate: $x_{k+1} = x_k + \alpha_k d_k$
		\State Set $k \leftarrow k + 1$
	\end{algorithmic}
\end{algorithm}

\newpage

\medskip\noindent
The steplength $\alpha_k$ is computed in Step 7 using a backtracking line search.
We refer to \eqref{relaxed} as the \emph{relaxed Armijo condition}. The term $\epsilon_R$ introduces a margin that facilitates the convergence analysis in the presence of noise, and as  discussed in Section~\ref{numerical}, is also useful in practice. Note that the line search cannot fail since \eqref{relaxed} is satisfied for sufficiently small $\alpha_k$, by definition of $\epsilon_R$.
In this paper, we assume that the quadratic subproblem \eqref{kkt}-\eqref{lincons} has a unique solution at every iteration---admittedly a strong assumption, but one that helps us focus on the effect of noise without the complicating effects of regularization parameters or trust regions. The study of a practical algorithm that employs those globalization strategies will be the subject of future work.


\section{Global Convergence} \label{global}
In this section we show that the iterates generated by Algorithm~\ref{algorithm} converge to a neighborhood of the solution determined by the noise level and certain characteristics of the problem. We also show that once the iterates reach this neighborhood they  cannot stray away from it (under normal circumstances). We start by stating the assumptions upon which our analysis is built.

\begin{assum} \label{assume1} The function $f$ has a Lipschitz continuous gradient with constant $L_f$. The functions  $\nabla c_i$ are Lipschitz continuous for $i=1,\ldots, m$ with the corresponding constants held in the vector $L_c$. 
 \end{assum}
 
 We also assume that the error (or noise) in the evaluation of the functions is bounded.
 
 \begin{assum} \label{assume2}
 There exist positive constants $\epsilon_f, \epsilon_c, \epsilon_g, \epsilon_J$ such that  for all $x \in \mathbb{R}^n$,
\begin{align}
|\tilde f(x) - f(x)|  \leq \epsilon_f , & \qquad
\|\tilde c(x) - c(x)\|_1  \leq \epsilon_c , \label{noise1} \\
\|\tilde g(x) - g(x)\|  \leq \epsilon_g , & \qquad
\|\tilde J(x) - J(x)\|_{1,2}  \leq \epsilon_J. \label{noise2}
\end{align}
\end{assum}

\noindent\medskip
Here, $\| \cdot \|$ denotes the Euclidean norm and $\|\cdot\|_{1,2}$ denotes the matrix norm induced by the $\ell_1$ norm on $\mathbb{R}^m$ and the Euclidean norm on $\mathbb{R}^n$.

\medskip
As already mentioned, we assume that, for all $k$,  the matrices $\tilde J_k$ have full rank so that the quadratic problem \eqref{kkt}-\eqref{lincons} has a unique solution. To state this precisely, we let $\sigma_{min} (A)$ denote the smallest singular value of a matrix $A$. 

\begin{assum} \label{assume3} For all $k$, the scalar $\beta_k$ in \eqref{hdef} satisfies  
\begin{equation}  \label{betab}
    0<  b_l \leq \beta_k \leq  b_u , 
\end{equation}
for some constants $b_l, b_u$,
and  there is a constant $\gamma >0$ such that
\begin{equation}  \label{singv}
    \sigma_{min} ( J_k) \geq \gamma , \quad \mbox{with   } \ \gamma > \epsilon_J , \ \ \forall k.
\end{equation}
Furthermore, the sequences ${\{f_k\}}, \{\|c_k\|\}$, $\{\|g_k\| \}$, $\{\|J_k \|\}$ generated by the algorithm are bounded.
\end{assum}
\medskip

By the matrix inversion lemma \cite{GoluvanL89} and \eqref{noise2}, if $J_k$ has full rank and $ \gamma > \epsilon_J$, then $\tilde J_k$ is also full rank and 
\begin{equation} \label{mudef}
     \| \tilde J_k^T( \tilde J_k \tilde J_k^T)^{-1} \| \leq \frac{1}{\gamma - \epsilon_J} \equiv \delta , \ \ \forall k.
 \end{equation}

The assumption that the sequences ${\{f_k\}}, \{\|c_k\|\}$, $\{\|g_k\| \}$, $\{\|J_k \|\}$ generated by the algorithm are bounded is fairly standard in the literature and is designed to avoid pathological situations. For example, the merit function $\phi$ may be unbounded below away from the solution if $\pi$  is not large enough. Although there are  strategies to avoid these situations (see e.g. \cite[\S 18.5]{mybook}, we do not include them in our algorithm, for simplicity. 
 \medskip

Given these three sets of assumptions, we are ready to study the convergence properties of Algorithm~\ref{algorithm}.
Let us apply the well known descent lemma (see e.g.\cite{bertsekas2015convex}) to the true (noiseless) merit function
\begin{equation} \label{realphi}
    \phi(x)= f(x) + \pi \| c(x)\|_1.
 \end{equation}
We have that for any $(x,d)$
\begin{align}  \label{plinear}
&\  \phi(x+\alpha d) \leq \phi(x) + \alpha g(x)^Td + \pi\big[\|c(x)+\alpha J(x)d\|_1 - \|c(x)\|_1\big] + \tfrac{1}{2}\big( L_f + \pi \|L_c\|_1 \big)\alpha^2\|d\|^2 .
\end{align}
Thus, we can write
\begin{align}  \label{bb}
 & \ \phi(x+\alpha d) - \phi(x) \leq  \ell(x;\alpha d) + \tfrac{1}{2}\big( L_f + \pi \|L_c\|_1 \big)\alpha^2\|d\|^2 ,
\end{align}
where 
\begin{equation} \label{linearm}
\ell(x;s)=g(x)^Ts +\pi\|c(x)+ J(x)s\|_1 - \pi \|c(x) \|_1.
\end{equation}
When function and derivatives are exact, it is easy to show that for $\pi$ sufficiently large and $\alpha$ sufficiently small we can guarantee a reduction in $\phi$; see \cite{mybook}. We must  establish that this is also the case in the noisy setting---before the iterates approach the region around the solution where noise dominates. We begin by establishing bounds on the step $d_k$.



\subsection{Preliminary results}
 The optimality conditions of the quadratic problem \eqref{kkt}-\eqref{lincons} are given by
\begin{align}   \label{kkt-qp}
\begin{pmatrix}
H_k& \tilde J_k^T\\
\tilde J_k & 0
\end{pmatrix}
\begin{pmatrix}
d_k \\
d_y
\end{pmatrix} = -
\begin{pmatrix}
\tilde g_k +  \tilde J_k^T y \\
\tilde c_k
\end{pmatrix} ,
\end{align}
for some Lagrange multiplier $y \in \mathbb{R}^m$. The step $d_k$ can be written as the sum of two orthogonal components,  
\begin{equation} \label{ddecomp}
   d_k = v_k + u_k ,
\end{equation}
where $v_k$ is in the range space of $\tilde J_k^T$ and $u_k$ is in the null space of $J_k$.
A simple computation from \eqref{kkt-qp} shows that
\begin{equation}  \label{uv}
v_k = -\tilde J_k^T(\tilde J_k \tilde J_k^T)^{-1}\tilde c_k, \quad \mbox{}\quad u_k=
 -\frac{1}{\beta_k}\tilde P_k \tilde g_k ,
\end{equation}
where 
\begin{equation} \label{projection}
\tilde P_k = I - \tilde J_k^T \left(\tilde J_k \tilde J_k^T\right)^{-1}\tilde J_k
\end{equation}
 is an orthogonal projection matrix onto the tangent space of the constraints. 
We now establish  bounds on $u_k, v_k$. In what follows, we let  $J^\dagger$ denote the Moore-Penrose generalized inverse of a matrix $J$, and define
$ P_k = I -  J_k^T \left( J_k  J_k^T\right)^{-1} J_k$. Since $\tilde P_k$ and $P_k$ are orthogonal projections, we have that $\| \tilde P_k \| = \| P_k\| =1$.

\begin{lem} \label{iris} 
Under Assumptions~3.1 and 3.2  we have both 
\begin{align}  
\|v_k\|_1  & \leq \delta \|\tilde c_k\|_1 \leq  {\delta} (\| c_k\|_1 + \epsilon_c) \label{vbound} \\
\|u_k\| &  \leq \frac{1}{\beta_k}\big(\|P_k g_k\| + \|g_k\|\eta\epsilon_J + \epsilon_g\big),             \label{tanupperbound}
\end{align}
where $\delta$ is defined in \eqref{mudef} and
\begin{equation}  \label{eta}
\eta=1/\gamma .
\end{equation}
Therefore, 
\begin{align}  
\|d_k\| \leq  \delta (\| c_k\|_1 + \epsilon_c)+ \frac{1}{\beta_k} \big(\|P_k g_k\| + \|g_k\|\eta\epsilon_J + \epsilon_g \big).   \label{dkbound}
\end{align}
\end{lem}

\begin{proof}
The bounds \eqref{vbound} follow directly from   \eqref{uv}, \eqref{mudef}, and \eqref{noise1}.
 By \eqref{noise2}, we can bound the norm of the tangential component as follows
\begin{align}
\|u_k\| & =\frac{1}{\beta_k}\|\tilde P_k\tilde g_k\|      \nonumber             \\
& \leq \frac{1}{\beta_k}\left(\|P_k g_k\| + \|(\tilde P_k - P_k)g_k\| + \|\tilde P_k\|\|g_k -\tilde g_k\| \right)  \nonumber\\ & \leq \frac{1}{\beta_k}\left(\|P_k g_k\| + \|(\tilde P_k - P_k)\| \| g_k\| + \epsilon_g \right) .  \label{simb}        
\end{align}
 Moreover, by the bounds on perturbed projection matrices 
 \cite[Theorems 2.3 and 2.4]{stewart1977perturbation}
we have that 
\begin{equation}  \label{marina}
 \| \tilde P_k - P_k\| \leq  \frac{\epsilon_J}{\gamma} \equiv \eta \epsilon_J .
\end{equation} 
This yields \eqref{tanupperbound}.
\end{proof}


\subsection{Penalty Parameter and Model Decrease}
We note from \eqref{bb} that in order to obtain a decrease in the true merit function $ \phi$, we must ensure that $\ell(x_k; \alpha_k d_k)$ is negative. We will see that this can be achieved for $\alpha_k=1$ by choosing a sufficiently large penalty parameter $\pi$, and provided noise does not dominate. 
\begin{lem}  \label{fullprof}
If at every iteration $k$ the penalty parameter satisfies 
\begin{equation} \label{pichoice}
   \pi_k \geq \frac{1}{1-\tau} \|  (\tilde J_k \tilde J_k^T)^{-1} \tilde J_k \tilde g_k\|_\infty, \quad \tau \in (0,1),
 \end{equation}
  then
\begin{align} \label{lchange}
\ell(x_k;d_k) \leq & -\frac{1}{\beta_k}g_k^T P_k g_k + \frac{1}{\beta_k}( \|g_k\|^2\eta\epsilon_J + \epsilon_g\|g_k\|) - \tau\pi_k \| c_k \|_1 +{\epsilon_g}{\delta}(\| c_k\|_1 +\epsilon_c)  \\
  &+\pi_k \left((2-\tau)\epsilon_c + \epsilon_J \left(\delta (\| c_k\|_1 + \epsilon_c)+ \frac{1}{\beta_k} (\|P_k g_k\| + \|g_k\|\eta\epsilon_J + \epsilon_g) \right) \right) . \nonumber
\end{align}
\end{lem}
\begin{proof}
Since 
$
 d_k = -\frac{1}{\beta_k}\tilde P_k\tilde g_k   -\tilde J_k^T(\tilde J_k \tilde J_k^T)^{-1}\tilde c_k ,
$
we have from \eqref{linearm}, \eqref{lincons}, \eqref{mudef}, \eqref{noise1}, \eqref{noise2},  and the definition of the $\| \cdot \|_{1,2}$ norm in \eqref{noise2}, that
\begin{align}
\ell(x_k;d_k) = & \ g_k^T d_k + \pi_k \| c_k + J_k d_k \|_1 - \pi_k \| c_k\|_1 \\
 \leq &-\frac{1}{\beta_k}g_k^T\tilde P_k \tilde g_k   -g_k^T\tilde J_k^T(\tilde J_k \tilde J_k^T)^{-1}\tilde c_k  +\pi_k\|c_k+ J_kd_k\|_1 - \pi_k \| c_k \|_1   \nonumber  \\
\leq& -\frac{1}{\beta_k}g_k^T\tilde P_k \tilde g_k   -g_k^T\tilde J_k^T(\tilde J_k \tilde J_k^T)^{-1}\tilde c_k  +\pi_k\|( c_k -\tilde c_k) +( J_k - \tilde J_k) d_k)\|_1 - \pi_k \| c_k \|_1      \nonumber \\
\leq & -\frac{1}{\beta_k}g_k^T\tilde P_k \tilde g_k  -\tilde g_k^T\tilde J_k^T(\tilde J_k \tilde J_k^T)^{-1}\tilde c_k  + {\epsilon_g}{\delta}\|\tilde c_k\|_1   
             +\pi_k (\epsilon_c + \epsilon_J\|d_k\|) - \pi_k \| c_k \|_1      \nonumber        \\
\leq & -\frac{1}{\beta_k}g_k^T\tilde P_k \tilde g_k  -\tilde g_k^T\tilde J_k^T(\tilde J_k \tilde J_k^T)^{-1}\tilde c_k  + {\epsilon_g}{\delta}(\| c_k\|_1 +\epsilon_c)  \nonumber   \\
&             +\pi_k \left[\epsilon_c + \epsilon_J \left(\delta (\| c_k\|_1 + \epsilon_c)+ \frac{1}{\beta_k} (\|P_k g_k\| + \|g_k\|\eta\epsilon_J + \epsilon_g) \right) \right] - \pi_k \| c_k \|_1,  \nonumber
\end{align}
the last line following by (\ref{dkbound}). 
Next,  since $\| \tilde P_k\|=1$ and recalling \eqref{marina}, we obtain
\begin{align*} 
 - g_k^T \tilde P_k \tilde g_k & \leq - g_k^TP_k g_k + \| g_k \| \|P_k g_k - \tilde P_k \tilde g_k \| \\
      & \leq - g_k^T  P_k  g_k + \| g_k \|  \| P_k g_k - \tilde P_k g_k + \tilde P_k g_k - \tilde P_k \tilde g_k \| \\
      & \leq - g_k^T  P_k g_k + \| g_k \|^2 \|P_k - \tilde P_k\| + \| g_k \| \|g_k - \tilde g_k \| \\
      & \leq - g_k^T P_k g_k + \|g_k \|^2 \eta \epsilon_J + \| g_k\| \epsilon_g .
 \end{align*}
 Therefore,
 \begin{align*}
 \ell(x_k; d_k) & \leq-\frac{1}{\beta_k}g_k^T P_k g_k + \frac{1}{\beta_k}( \|g_k\|^2\eta\epsilon_J + \epsilon_g\|g_k\|)  -\tilde g_k^T\tilde J_k^T(\tilde J_k \tilde J_k^T)^{-1}\tilde c_k  + {\epsilon_g}{\delta}(\| c_k\|_1 +\epsilon_c)  \nonumber   \\
&             +\pi_k \left[\epsilon_c + \epsilon_J \left(\delta (\| c_k\|_1 + \epsilon_c)+ \frac{1}{\beta_k} (\|P_k g_k\| + \|g_k\|\eta\epsilon_J + \epsilon_g) \right) \right] - \pi_k \| c_k \|_1.  
\end{align*}
Now suppose that we choose the parameter $\pi_k$ so that \eqref{pichoice} holds.  
Then 
\[
-\tilde g_k^T\tilde J_k^T(\tilde J_k \tilde J_k^T)^{-1}\tilde c_k \leq \|\tilde g_k^T\tilde J_k^T(\tilde J_k \tilde J_k^T)^{-1}\|_\infty\|\tilde c_k\|_1 \leq (1-\tau)\pi_k (  \|c_k\|_1+ \epsilon_c) ,
\]
and it follows that
\begin{align*} 
\ell(x_k;d_k) \leq & -\frac{1}{\beta_k}g_k^T P_k g_k + \frac{1}{\beta_k}( \|g_k\|^2\eta\epsilon_J + \epsilon_g\|g_k\|) - \tau\pi_k \| c_k \|_1 +{\epsilon_g}{\delta}(\| c_k\|_1 +\epsilon_c)  \\
  &+\pi_k \left[(2-\tau) \epsilon_c + \epsilon_J \left(\delta (\| c_k\|_1 + \epsilon_c)+ \frac{1}{\beta_k} (\|P_k g_k\| + \|g_k\|\eta\epsilon_J + \epsilon_g) \right) \right]  .
\end{align*}
\end{proof}

Lemma~\ref{fullprof} implies that for any $x_k $ such that the right hand side of \eqref{lchange} is negative, we have
$
\ell(x_k;d_k)<0.
$
We  now provide conditions under which the decrease in $\ell$ is proportional to the optimality conditions of the nonlinear problem \eqref{problem}. Specifically, since $g_k^TP_k g_k = \| P_k g_k\|^2$ is the norm squared of the projected gradient, a combination of $g_k^TP_k g_k$ and $\| c_k\|_1$ can be regarded as a measure of  stationarity of the constrained optimization problem. The following result assumes that the optimality measure is not small compared to the errors (or noise).

\begin{cor} \label{suff}
Choose  any $\theta_1 \in [0,1)$.
For any $x_k $ sufficiently far from the solution such that 
\begin{align}
(1-\theta_1)\Big(\frac{1}{\beta_k}g_k^T P_k g_k + &\tau\pi_k \| c_k \|_1\Big)  \geq  
{E}(x_k,\beta_k, \pi_k)  ,   \label{critical1} 
\end{align} 
where  
\begin{align}
{E}(x,\beta,\pi) =& 
\frac{1}{\beta}( \|g(x)\|^2\eta\epsilon_J + \epsilon_g\|g(x)\|)  +{\epsilon_g}{\delta}{(\| c(x)\|_1 + \epsilon_c)}   \nonumber  \\
  +& \pi \left[{(2-\tau)\epsilon_c}  + \epsilon_J \left(\delta (\| c(x)\|_1 + \epsilon_c)+ \frac{1}{\beta} (\|P(x) g(x)\| + \|g(x)\|\eta\epsilon_J + \epsilon_g) \right)   \right] , \label{epsilonbar} 
\end{align} 
we have 
\begin{equation}  \label{ldecrease}
\ell(x_k;d_k) \leq  -\theta_1\left( \frac{1}{\beta_k}g_k^T P_k g_k +\tau\pi_k \| c_k \|_1 \right).
\end{equation}
\end{cor} 
\begin{proof}
For any $\theta_1 \in [0,1)$, we can rewrite (\ref{lchange}) as
\begin{align} 
\ell(x_k;d_k) \leq & -\theta_1( \frac{1}{\beta_k}g_k^T P_k g_k +\tau\pi_k \| c_k \|_1)- (1-\theta_1)(\frac{1}{\beta_k}g_k^T P_k g_k +\tau\pi_k \| c_k \|_1) \nonumber \\
 &+\frac{1}{\beta_k}( \|g_k\|^2\eta\epsilon_J + \epsilon_g\|g_k\|)  
  +{\epsilon_g}{\delta}\|{(\| c_k\|_1 + \epsilon_c)}  \nonumber \\
  & +\pi_k \left({(2-\tau)\epsilon_c}  + \epsilon_J \left(\delta (\| c_k\|_1 + \epsilon_c)+ \frac{1}{\beta_k} (\|P_kg_k\| + \|g_k\|\eta\epsilon_J + \epsilon_g) \right)  \right) ,\nonumber
\end{align}
from which \eqref{ldecrease} follows by condition \eqref{critical1}.
\end{proof}

{In order to make this result, and similar results to be proved later, more understandable and more convenient to use, we recall that $g(x)^TP(x)g(x) = \|P(x) g(x)\|^2$, and define the function
\begin{equation}   \label{psi}
   \psi_\pi(x) = \frac{1}{b_u}\|P(x) g(x)\|^2  +\pi \tau \| c(x) \|_1 ,
\end{equation}
where $b_u$ is given in \eqref{betab}.
Clearly, $\psi_\pi$ may be viewed as a measure of non-stationarity since $\psi_\pi(x^\ast)=0$ when $x^\ast$ is a stationary point of the problem (\ref{problem}).
Given this notation we can restate a slightly weaker version of Corollary~\ref{suff}.

\begin{cor} \label{suff2}
Choose  any $\theta_1 \in [0,1)$.
For any $x_k $ sufficiently far from the solution such that 
\begin{align}
\psi_{\pi_k}(x_k)  \geq  E(x_k,\beta_k, \pi_k)/(1-\theta_1),   \label{critical11} 
\end{align} 
we have 
\begin{equation}  \label{lde2}
\ell(x_k;d_k) \leq  -\theta_1\left( \frac{1}{\beta_k}g_k^T P_k g_k +\tau\pi_k \| c_k \|_1 \right) 
\leq -\theta_1  \psi_{\pi_k}(x_k).
\end{equation}
\end{cor} 
\begin{proof} The result follows from the fact that
\[
  \psi_{\pi_k}(x_k) \leq \Big(\frac{1}{\beta_k}g_k^T P_k g_k + \tau\pi_k \| c_k \|_1\Big).
  \]
  \end{proof}

\subsection{Line search}

Since $\pi_k$ is defined by \eqref{updatepi} and \eqref{linden}, and by Assumptions~\ref{assume3}, we have that $\{ \|\hat \lambda_k\| \}  $ is bounded.
Moreover, since $\{ \pi_k \}$ is monotone and since $\pi_k - \pi_{k-1}$ is either zero or greater than $\pi_{k-1}$, there exists values $k_0$ and $\bar \pi$ such that:
\begin{equation}  \label{k0}
   \pi_k= \bar \pi, \quad \forall k \geq k_0,
\end{equation}
and (\ref{pichoice}) is satisfied. The rest of the analysis assumes that the penalty parameter has attained that fixed value $\bar \pi$. Thus, for the rest of the section
\begin{equation} \label{fixphi}
    \tilde \phi(x_k) \equiv \tilde f(x_k) + \bar \pi \|\tilde c(x_k)\|_1, \qquad
     \phi(x_k) \equiv  f(x_k) + \bar \pi \| c(x_k)\|_1, \ \forall k \geq k_0.
  \end{equation}
  
Algorithm~\ref{algorithm} sets $x_{k+1}=x_k+\alpha_k  d_k$, where $\alpha_k$ is chosen by repeated halving until the  {\em relaxed  Armijo condition} is satisfied:
\[
\tilde \phi(x_k+\alpha_k  d_k) \leq \tilde \phi(x_k) +{\nu} \alpha_k \tilde \ell(x_k; d_k) + \epsilon_R, 
\]
for some constants
$
{\nu} \in (0,1)$ and $ \epsilon_R \,{\geq} \, 2( \epsilon_f + \bar \pi \epsilon_c ),
$
where $\tilde \ell(x_k;d_k)$ is defined in \eqref{lineart}.
%
In other words, we require that the decrease in the noisy merit function be a fraction of the decrease of the noisy first-order model $\tilde \ell$, plus a relaxation term. 

To ensure that the line search yields significant progress toward a solution, we need to show that $\alpha_k$ is bounded away from zero and that $\tilde \ell(x_k;d_k)$ is sufficiently negative. To do so, we recall that we
    have established in \eqref{lde2}  that the noiseless first-order model $\ell(x_k;d)_k$ is sufficiently negative when condition \eqref{critical11} is satisfied. To relate $ \ell(x_k;d_k)$ to $\tilde \ell(x_k;d_k)$, we recall \eqref{linearm} and \eqref{lineart}, and measure the difference between these two quantities.
By (\ref{dkbound}) 
\begin{align}
|\tilde \ell(x_k;d_k) - \ell(x_k;d_k)| \leq & \ \epsilon_g \|d_k\| +2 \bar \pi \epsilon_c+ \bar \pi \epsilon_J \|d_k\|  \nonumber \\
\leq & \ ( \epsilon_g +\bar \pi \epsilon_J) \left( \delta (\| c_k\|_1 + \epsilon_c)+ \frac{1}{\beta_k} (\|P_k g_k\| + \|g_k\|\eta\epsilon_J + \epsilon_g) \right) + 2 \bar \pi \epsilon_c 
 \label{bwon0} \\
 \leq & \ {( \epsilon_g +\bar \pi \epsilon_J) \left( \delta (C_c + \epsilon_c)+ \frac{1}{b_l} (C_g + C_g \eta\epsilon_J + \epsilon_g) \right) + 2 \bar \pi \epsilon_c }
 \nonumber \\
\equiv & \    {\epsilon_\ell  ,} \label{bwon}
\end{align}
where  $C_g, C_c$ are constants such that 
\begin{equation} \label{Cs}
     \|g(x_k) \| \leq C_g, \qquad \|c(x_k)\|_1 \leq C_c \quad \forall k> k_0.
\end{equation}

We know that these constants exist because of Assumption~\ref{assume3}. 
We now describe conditions under which one can characterize the size of the steplength $\alpha_k$.
Let
\begin{equation}   \label{ltotal}
L= L_f + \bar \pi \|L_c\|_1,
\end{equation}
where $L_f, L_c$ are defined in Assumptions~\ref{assume1}.

\begin{thm}   \label{mel}
Let $\theta_1$ be defined as in Corollary~\ref{suff},
 choose constants $\theta_2<\theta_1$,  $\nu \in (0,1)$ and
\begin{equation} \label{erdef}
 \epsilon_R \, {\geq} \, 2(\epsilon_f + \bar \pi \epsilon_c) \equiv 2 \epsilon_\phi.
\end{equation}
 Then, for all iterates $x_k$ {with $k\geq k_0$ } that satisfy both \eqref{critical11} and 
\begin{equation}  \label{critical2}
 (1-\nu)(\theta_1-\theta_2)  \left(\frac{1}{\beta_k} \|P_k g_k\|^2 +{\bar\pi}\tau\|c_k\|_1 \right)  >2\nu  \epsilon_\ell,
\end{equation} 
 if the steplength satisfies
 \begin{equation}  \label{alphabound}
	\alpha_k < \frac{(1-\nu) \theta_2 \left( {\frac{1}{\beta_k} }\|P_kg_k\|^2 +\bar \pi\tau\|c_k\|_1  \right) }{\frac{L}{2} [  \delta^2 (\|c_k\|_1+ \epsilon_c)^2 +  \frac{1}{\beta_k^2}(\|P_kg_k\| + \|g_k\|\eta\epsilon_J + \epsilon_g)^2 ] }  \equiv \hat{\alpha}_k,
\end{equation}
then
\begin{equation} \label{armijo-R}
\tilde \phi(x_k+\alpha_k  d_k) \leq \tilde \phi(x_k) +\nu \alpha_k \tilde \ell(x_k; d_k) + \epsilon_R .
\end{equation}
\end{thm}
\begin{proof}
By  the definition \eqref{erdef} of $\epsilon_\phi$,  \eqref{bb}, \eqref{ltotal},  {the convexity of $ \ell (x_k; \cdot)$, (\ref{bwon}),}
 (\ref{lde2}), the fact that  $P_k^2=P_k$, and \eqref{dkbound},  we get
\begin{align*}
\tilde \phi(x_k+\alpha d_k) - \tilde \phi(x_k) \leq & \, \phi(x_k+\alpha d_k) - \phi(x_k) + 2 \epsilon_\phi   \\
 \leq & \,   \ell (x_k;\alpha d_k) + 2 \epsilon_\phi+ \tfrac{L}{2}\alpha^2\|d_k\|^2  \\
 \leq & \,  \alpha \ell (x_k; d_k) + 2 \epsilon_\phi+ \tfrac{L}{2}\alpha^2\|d_k\|^2  \\
= & \, {\nu \alpha  \ell(x_k; d_k) + 2 \epsilon_\phi + (1-\nu) \alpha \ell(x_k;d_k)   + \tfrac{L}{2}\alpha^2\|d_k\|^2}   \\
\leq & \, {\nu \alpha  \tilde \ell(x_k; d_k)+\nu \alpha \epsilon_{\ell}  + 2 \epsilon_\phi+ (1-\nu) \alpha \ell(x_k;d_k)   + \tfrac{L}{2}\alpha^2\|d_k\|^2}  \\
 \leq &\,  \nu \alpha  \tilde \ell (x_k; d_k) +2 \epsilon_\phi+2\nu \alpha \epsilon_{\ell} - (1-\nu) 
\theta_1 \alpha \Big(\frac{1}{\beta_k} g_k^TP_k g_k +{\bar\pi}\tau\|c_k\|_1 \Big)      \\
& + \tfrac{L}{2}\alpha^2\|d_k\|^2 \\
\leq & \, \nu \alpha \tilde \ell (x_k; d_k) +2 \epsilon_\phi+2\nu \alpha \epsilon_{\ell} - (1-\nu) 
\theta_1 \alpha \Big(\frac{1}{\beta_k} \|P_k g_k\|^2 +{\bar\pi}\tau\|c_k\|_1 \Big) \\
 &+\tfrac{L}{2}\alpha^2 \Big[  \delta^2 (\|c_k\|_1+ \epsilon_c)^2 +  \frac{1}{\beta_k^2}\big(\|P_k g_k\| + \|g_k\|\eta\epsilon_J + \epsilon_g\big)^2 \Big] ,
\end{align*}
the last line following from the orthogonality of the components \eqref{ddecomp} of $d_k$.

Now we choose a constant $\theta_2<\theta_1$, and consider iterates $x_k$ such that \eqref{critical2} holds.
For such iterates we have,
\begin{align*}
\tilde \phi(x_k+\alpha d_k) {- \tilde \phi(x_k) }\leq & \nu \alpha  \tilde \ell(x_k; d_k) +2 \epsilon_\phi - (1-\nu) 
\theta_2 \alpha \Big( { \frac{1}{\beta_k} }\|P_k g_k\|^2 +\bar\pi\tau\|c_k\|_1  \Big) \\
 &+\tfrac{L}{2}\alpha^2 \left[  \delta^2 (\|c_k\|_1+ \epsilon_c)^2 +  \frac{1}{\beta_k^2}\Big(\|P_k g_k\| + \|g_k\|\eta\epsilon_J + \epsilon_g\Big)^2 \right] .
\end{align*}

Then, for any steplength satisfying \eqref{alphabound}
where $x_k$ satisfies the \eqref{critical1} and \eqref{critical2}, 
we have
\[
\tilde \phi(x_k+\alpha d_k) - \tilde \phi(x_k) \leq  \nu \alpha  \tilde l(x_k;d_k) +2 \epsilon_\phi ,
\]
and thus (\ref{armijo-R}) holds since $\epsilon_R \,{\geq 2}\epsilon_\phi$.
\end{proof}
 
Note that condition \eqref{critical2} is implied by the slightly weaker inequality
\begin{equation}   \label{critical22}
(1-\nu)(\theta_1-\theta_2)  \psi_{\bar \pi}(x_k)  >2\nu  \epsilon_\ell.
\end{equation}

Since the numerator in \eqref{alphabound} is bounded away from zero by \eqref{critical2}, and the denominator is bounded above given the assumed global upper bounds on $c_k, g_k,$ and lower bound on $\beta_k$ stated in Assumptions~\ref{assume3}, it follows     that  there is a constant $\bar \alpha$ such that $\hat \alpha_k > 2 \bar \alpha$ for all $k \geq k_0$. The algorithm employs a backtracking line search that halves each trial step, hence we  can conclude that
\begin{equation}  \label{lowalpha}
       \bar \alpha \leq \alpha_k, \qquad \mbox{for} \ k \geq k_0.
\end{equation}
This will allow us to show that, when the conditions in Theorem~\ref{mel} are satisfied, the algorithm will make non-negligible progress.

 \subsection{The Main Convergence Result}
Now we show that  Algorithm~\ref{algorithm} will eventually generate iterates close to a stationary point of the problem, as measured by the function $\psi_{\bar \pi}(x)$ defined in \eqref{psi}. To do so, we note that condition \eqref{critical11} implies that the {linear model decrease $\ell$ is sufficiently negative in the sense of \eqref{lde2},} and we have established a bound in \eqref{bwon} for the distance between $\ell$ and $\tilde \ell$. Furthermore, we have shown that condition \eqref{critical22} ensures that the relaxed Armijo condition  \eqref{armijo-R} is satisfied for steplengths $\alpha_k$ that are bounded away from zero. Those two conditions---\eqref{critical11}, \eqref{critical22}---are necessary to ensure that the algorithm makes significant progress,  but they are not sufficient. To control the { effect of noise in testing \eqref{relaxed} as well as the {effect of the} relaxation factor}, we  impose one additional condition,
\begin{equation}   \label{critical3}
\psi_{\bar \pi}(x_k) \geq \frac{2\epsilon_R + 4 \epsilon_\phi }{ \nu \bar \alpha \theta_2} ,
\end{equation}
to help define the region where Algorithm~\ref{algorithm} progresses toward stationarity. 

One more refinement is needed. The definition of the term $E(x_k, \beta_k, \pi_k)$ defined in \eqref{epsilonbar} involves $c(x_k)$ and $g(x_k)$, which makes the region defined by \eqref{critical11} difficult to interpret. Therefore, we compute an upper bound for $E$. If we define
\begin{align} 
{\cal E} =& 
\frac{1}{b_l}( C_g^2\eta\epsilon_J + \epsilon_gC_g)  +{\epsilon_g}{\delta}{(C_c + \epsilon_c)}   \nonumber  \\
  +& \bar \pi \left[{(2-\tau)\epsilon_c}  + \epsilon_J \left(\delta (C_c + \epsilon_c)+ \frac{1}{b_l} (C_g + C_g\eta\epsilon_J + \epsilon_g) \right)   \right] ,\label{ebar}
\end{align}
where $C_g, C_c$ are given in \eqref{Cs},
then we have that ${E}(x_k,\beta_k, \bar \pi) \leq {\cal E}$ for all $k \geq k_0$.
 We can thus state a  condition that implies \eqref{critical11}:
\begin{equation}   \label{critical111}
\psi_{{\bar \pi}}(x_k)  \geq  \frac{{\cal E}}{(1-\theta_1)}, \qquad  \forall k\geq k_0.
\end{equation}

In summary, the analysis presented above holds if conditions \eqref{critical111}, \eqref{critical22} are satisfied and we also impose condition \eqref{critical3}. This allows us to characterize a region, which we denote by $C_1$, where errors dominate and improvement in the  merit function $\phi$ cannot be guaranteed. In other words, $C_1$ is  the region where at least {one} of the three conditions---\eqref{critical111}, \eqref{critical22}, \eqref{critical3}---is not satisfied. 
{
\begin{Def}  \label{defcritical}
The critical region $C_1$ is defined as the set of $x \in \mathbb{R}^n$ satisfying
\begin{align}
\psi_{\bar \pi}(x)  \leq 
 & \max{ 
 \left\{ \frac{{\cal E} }{(1-\theta_1)}, \frac{2{\nu} \epsilon_{\ell}}{({1-\nu})(\theta_1-\theta_2) }, \frac{2\epsilon_R + 4 \epsilon_\phi }{ \nu \bar \alpha \theta_2} \right\}, \label{kind}
}
\end{align}
where $ \cal E$ and $\epsilon_\ell$ are defined by \eqref{ebar} and \eqref{bwon}, respectively, and $\theta_1, \theta_2$ are constants such that
$0 <\theta_2 < \theta_1 <1$.
\end{Def}

\medskip
We also define the following   set.

\begin{Def}  \label{supremum}
Let $w=\sup\{\phi(x): x\in C_1\}$, and define the level set 
\[
	C_2= \{x: \phi(x)\leq w +2\epsilon_\phi +\epsilon_R\}  .
\]
\end{Def}
 \medskip
Note that by construction $C_1 \subseteq C_2$.
We are now ready to state the main convergence result. 
\begin{thm}  \label{mainth}
Suppose that Algorithm~\ref{algorithm}  generates a sequence $\{x_k\}$ from $x_0$ satisfying  Assumptions~\ref{assume1}-\ref{assume3}. 
There is an iteration $k_1$ at which $\{x_k\}$  enters the critical region $C_1$, and for all $k >k_1$ the iterates remain in the critical {level set} $C_2$. The iterates may leave $C_1$,   but there must be infinitely many iterates in $C_1$.  
\end{thm}
\begin{proof}
%
Recall that the index $k_0$ is defined in \eqref{k0}.
If $k \not \in C_1$ and $k \geq k_0$, then the assumptions of Theorem~\ref{mel} are satisfied and \eqref{armijo-R} holds. Therefore, by
\eqref{lowalpha}, \eqref{bwon}, \eqref{ldecrease},  \eqref{lde2}, \eqref{psi}, \eqref{critical2}
\begin{align}
 \phi(x_k+\alpha_k d_k) -  \phi(x_k) \leq &\,  {\tilde \phi(x_k+ \alpha_k d_k) -  \tilde \phi(x_k) +  2\epsilon_\phi} \\
\leq &\, \nu \bar \alpha  \tilde \ell_\phi(x_k; d_k) + 2\epsilon_\phi + \epsilon_R\\
\leq  &\, \nu  \bar \alpha  \ell_\phi(x_k; d_k)+ \nu \bar \alpha \epsilon_{\ell} +2 \epsilon_\phi +  \epsilon_R \nonumber \\
\leq  &-\nu \bar \alpha \theta_1\big( \frac{1}{\beta_k}g_k^T P_k g_k +\tau \bar \pi \| c_k \|_1\big) + \nu \bar \alpha \epsilon_{\ell} + 2 \epsilon_\phi +\epsilon_R  \nonumber  \\
\leq  &-\nu \bar \alpha \theta_1 \psi_{\bar \pi}(x_k)+ \nu \bar \alpha \epsilon_{\ell} + 2 \epsilon_\phi +\epsilon_R  \nonumber  \\
=  &-[\nu \bar \alpha \theta_2 + \hat\alpha \nu (\theta_1-\theta_2)]  \psi_{\bar \pi}(x_k) + \nu \bar \alpha \epsilon_{\ell} + 2 \epsilon_\phi +\epsilon_R  \nonumber  \\
\leq  &-\nu \bar \alpha \theta_2 \psi_{\bar \pi}(x_k)   + 2 \epsilon_\phi  + \epsilon_R .\label{lsdecrease}
\end{align} 
Combining this bound with \eqref{critical3}, we have that if $x_k \notin C_1$ then 
\begin{equation} \label{bikn}
\phi(x_{k+1}) -  \phi(x_k)  \leq -\frac{\nu \bar \alpha \theta_2}{2} \psi_{\bar \pi}(x_k) .
\end{equation}
Since the sequence $\{ \phi (x_k) \}$ is bounded below by Assumptions~\ref{assume3}, $\psi_{\bar \pi}(x_k) $ converges to zero and thus it follows that  Algorithm~\ref{algorithm} eventually generates an iterate in $C_1$.

Now if $x_k \in C_1$, then by Step~6 in Algorithm~\ref{algorithm},  
$ \phi(x_{k+1}) \leq \phi(x_k) +2\epsilon_\phi +\epsilon_R \leq w+2\epsilon_\phi +\epsilon_R$, so that  $x_{k+1} \in C_2.$ 

On the other hand, if $x_k \in C_2 \setminus C_1$, then by \eqref{bikn}
\[
\phi(x_{k+1}) -  \phi(x_k)  \,  \leq \,  0 ,
\]
which implies $x_{k+1} \in C_2$. 
Thus the rest of the sequence lies in $C_2$, with infinitely many iterates in $C_1$.
\end{proof}

We should note that since we are not assuming that the objective function is strongly convex or satisfies a quadratic growth condition, it is possible that the supremum in Definition~\ref{supremum} is $w=\infty$. This is, however, an unlikely scenario.

\subsection{Discussion}

Let us take a closer look the main result of this paper, Theorem~\ref{mainth}, since  the critical region $C_1$ defined in \eqref{kind} is complex. 

By the definitions \eqref{ebar} and \eqref{bwon}, we have that  $\cal E$ and $\epsilon_\ell$ are both of order $O(\epsilon_c, \epsilon_g,\epsilon_J$), and so is the right hand side in \eqref{kind}. This is as desired. The constants in these orders of magnitude matter, so we must characterize them.

First note that the critical region $C_1$, the set $ C_2$ and $\bar \pi$ depend on the starting point $x_0$. It is then possible that $\bar \pi$ could be very large in some cases, although in practice this does not seem to be a major concern. The constants $C_g, C_c$, which also enter in the definition of $\cal E$ and $\epsilon_\ell$ could  be quite large. One can, however, give a tighter definition of $C_1$ by not introducing these constants. In this case, we would define $\epsilon_\ell$ by \eqref{bwon0} and employ \eqref{critical11}, rather than \eqref{critical111}. This makes the main theorem more precise, albeit more difficult to interpret.

Returning to the constants in \eqref{ebar} and \eqref{bwon}, we have that
\[
 \epsilon_\ell, {\cal E}  \sim \left[\delta, \, \frac{1}{b_l}, \, \frac{\eta}{b_l}\right] , 
\]
and from \eqref{singv}, \eqref{mudef}, \eqref{eta} we observe that
\[ 
\sigma_{min}(\tilde J_k) \geq \gamma, \qquad
 \delta= \frac{1}{\gamma- \epsilon_J} \geq \frac{1}{\sigma_{min}(\tilde J_k) -\epsilon_J}, \qquad \mbox{and} \
 \eta=\frac{1}{\gamma}= \frac{1}{\sigma_{min}(\tilde J_k)}.
 \]
The effect of a near rank-deficient Jacobian and  Hessian approximations $\beta_k I$ are now apparent.

 It is interesting to compare $C_1$ with the region obtained by Berahas et al. \cite{berahas2019derivative} for unconstrained strongly convex optimization. {When constraints are not present, i.e., $m=0$, conditions \eqref{critical1} and \eqref{critical2} defining $C_1$ reduce to requirements of form $\|g_k\|\geq c_1\epsilon_g$ and $\|g_k\|^2\geq c_2 (\epsilon_g\|g_k\|+\epsilon_g^2)$ for some constants $c_1$ and $c_2$, respectively.  That  corresponds to \emph{Case 1} in the analysis of \cite{berahas2019derivative}, in which case $\epsilon_g$ is small as compared to $\|g_k\|$ by some factor $\beta\in(0,1)$, so that the line search ensures an improvement in the exact objective function -- $f(x)$ in our notation.  Similar to the setting in this paper, \cite{berahas2019derivative} employs a relaxed line search which does not fail even in the critical region; that is, when $\|g_k\|\leq \beta \epsilon_g$.  Their analysis then provides a level set that the iterates cannot leave, which depends on the relaxation term $\epsilon_R$ as well as $\epsilon_\phi$ (i.e. $\epsilon_f$ in the unconstrained case) as in the definition of $C_2$ in our analysis.  Since strong convexity is assumed in \cite{berahas2019derivative}, they can define this level set in terms of a strong convexity parameter rather than a bound such as $w$ in Definition~\ref{supremum}.}

%
%
%


\section{Numerical Experiments}   \label{numerical}

We implemented Algorithm~\ref{algorithm} in Python.  We set $\nu=0.1$, $\tau=0.9$, and $\beta_k=50$, for all $ k$.  The purpose of the numerical experiments is to supplement the theoretical results, which are stated in terms of the merit function $\phi$, by reporting the distance to the solution $\| x_k - x^\ast\|$ as the iteration progresses. In order to gain an idea of this behavior, it suffices to test only a few examples. We selected the following three small-scale equality-constrained problems from the CUTEst set \cite{gould2015cutest}.

\begin{table}[htp]
\centering
\begin{tabular}{ r | l | l | l  }
\hline
problem & classification & objective & constraints\\
\hline
HS7 & OOR2-AN-2-1 & $\ln(1+x_1^2)-x_2$ & $(1+x_1^2)^2+x_2^2=4$\\
\hline
BT11 & OOR2-AN-4-3 & $-x_1x_2x_3x_4$ & $\begin{aligned}x_1^3 + x_2^2 &=1\\ x_1^2x_4 - x_3 &= 0\\ x_4^2-x_2 &= 0\end{aligned}$ \\
\hline
HS40 & OOR2-AY-5-3 & $\begin{aligned}(x_1-1)^2 + (x_1-x_2)^2 + (x_2-x_3)^2\\ +(x_3-x_4)^4 +(x_4-x_5)^4 \end{aligned}$ & $\begin{aligned} x_1+x_2^2+x_3^3 &= -2+\sqrt{18}\\ x_2+x_4+x_3^2 &= -2+\sqrt{8}\\ x_1-x_5&=2 \end{aligned}$ \\
\hline
\end{tabular}
\end{table}

\bigskip\noindent
We add uniformly distributed random noise to the exact function values and to each component of the exact gradients; i.e., for $\xi_i \sim \mathcal{U}(-\epsilon_1,\epsilon_1)$, and $\psi_{ij}\sim \mathcal{U}(-\epsilon_2,\epsilon_2)$ we set
\begin{align*}
\tilde f(x) &= f(x)+\xi_0,  \qquad ~~~\tilde c_i(x) = c_i(x)+\xi_i\\
\tilde g_i(x) &= g_i(x)+\psi_{0j}, \qquad \tilde J_{ij}(x) = J_{ij}(x)+\psi_{ij}.
\end{align*} 
In our tests, we vary $\epsilon_1, \epsilon_2$, and report  $\|x_k-x^\ast\|$, where $x^\ast$ is a locally optimal solution obtained by using exact gradients in the algorithm. For each of these problems, $x^\ast$ is a a nondegenerate stationary point.

\paragraph{Asymptotic Behavior.}
  In Figure~\ref{fig:converge}, we plot $\|x_k-x^\ast\|$ for 1000 iterations, for $\epsilon_1=\epsilon_2=10^{-3}$ in the definitions of $\xi_i $, and $\psi_{ij}$
  We also display the values of $\epsilon_f, \epsilon_c, \epsilon_g, \epsilon_J$ defined in \eqref{noise1}-\eqref{noise2} We should note that in each of the runs the penalty parameter $\pi_k$ became fixed within the first 15 iterations. We observe that $\{\|x_k -x^\ast\|\}$ is contained in a band whose upper bound is frequently visited by the algorithm, whereas the lower bound is defined by large irregular spikes. These results suggest that if one desires the highest accuracy in the solution, the algorithm should continue beyond the point where oscillations in the merit function occur, since there is little risk that the iterates will stray away from the neighborhood of the solution, and there is a chance that significantly higher accuracy is achieved at some iterates.

\begin{figure}[H]
\caption{Distance to optimality ($\log_2(\|x_k-x^\ast\|)$) vs iteration number for $\epsilon_1=\epsilon_2=10^{-3}$}
\centering
\begin{subfigure}{\textwidth}
\includegraphics[width=\textwidth]{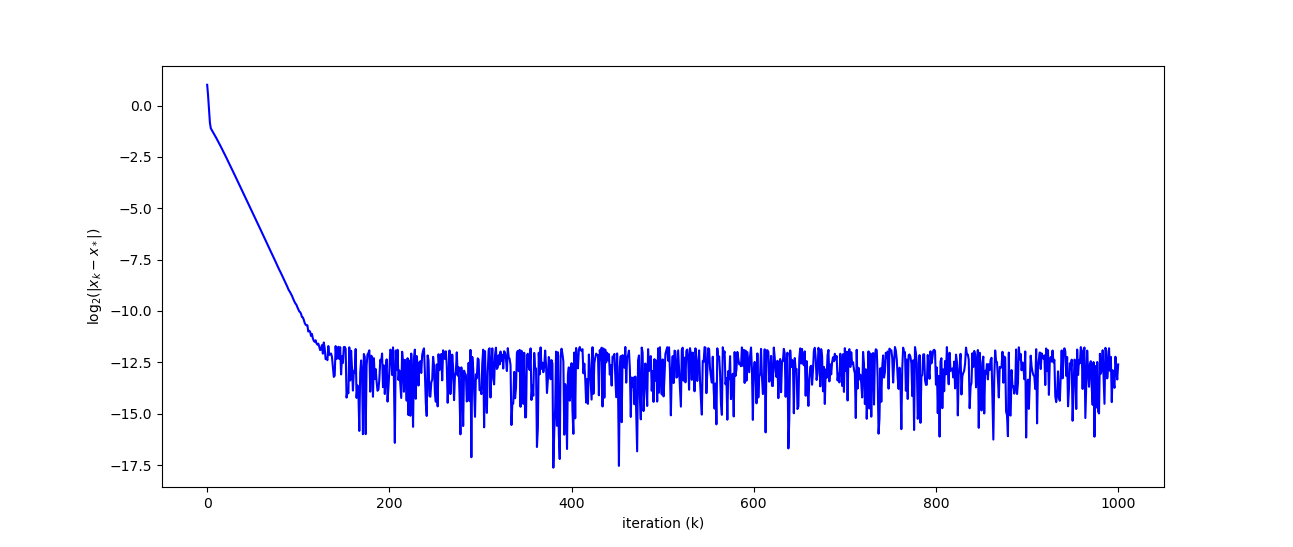}
\caption{HS7.  $\epsilon_f=10^{-3}, \epsilon_c=10^{-3}, \epsilon_g=1.41\times 10^{-3}, \epsilon_J=1.41\times 10^{-3}$ }
\end{subfigure}\\
\begin{subfigure}{\textwidth}
\includegraphics[width=\textwidth]{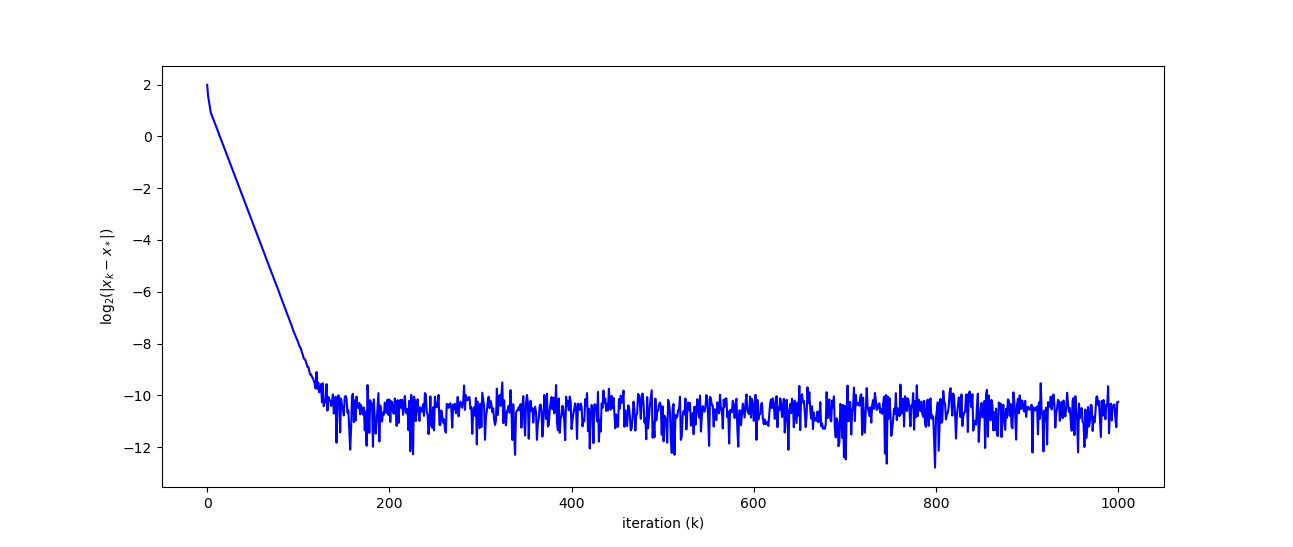}
\caption{BT11.  $\epsilon_f=10^{-3}, \epsilon_c=3\times 10^{-3}, \epsilon_g=2.24\times 10^{-3}, \epsilon_J=6.71\times 10^{-3}$}
\end{subfigure}
\begin{subfigure}{\textwidth}
\includegraphics[width=\textwidth]{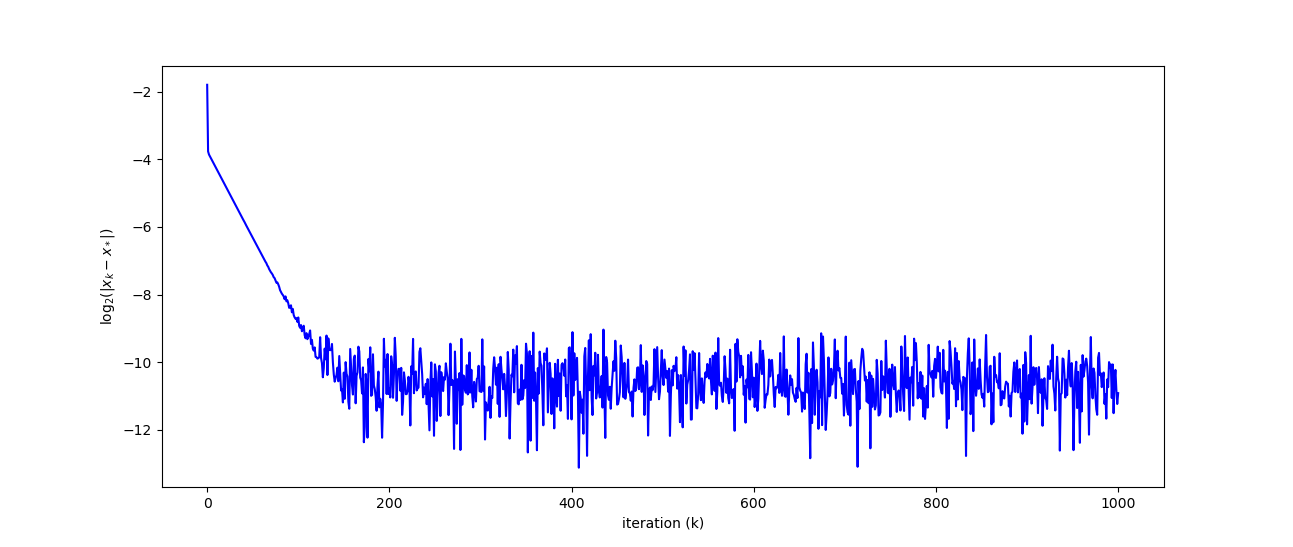}
\caption{HS40.  $\epsilon_f=10^{-3}, \epsilon_c=3 \times 10^{-3}, \epsilon_g=2\times 10^{-3}, \epsilon_J=6\times 10^{-3}$}
\end{subfigure}
\label{fig:converge}
\end{figure}

\bigskip


\paragraph{Benefits of the relaxed line search.}  The only unconventional part of Algorithm~\ref{algorithm} is the relaxed line search \eqref{relaxed}.  To observe the effect of the relaxation, we solved the test problems with and without it; the results are reported in Tables~\ref{relaxation1}--\ref{relaxation3}.  We observe that when the relaxation is disabled, the line search often fails in a neighborhood of $x^\ast$ (we terminate the algorithm as soon as there is a line search failure).  When the relaxation is enabled, the line search is always successful.  In this case, we let the algorithm run for 100, 500, and 1000 iterations.   
It is apparent that the relaxed line search allows the algorithm to continue iterating past the point where the traditional line search would fail, yielding much better accuracy in the solution. 

\medskip

\begin{table}[h]
\caption{$\min_k\{\|x_k-x^\ast\|\}$ when $\epsilon_1=\epsilon_2=10^{-5}$}
\begin{tabular}{r | c | c | c | c | c }
\hline
& \multicolumn{2}{c}{relaxation disabled} & \multicolumn{3}{|c}{relaxation enabled} \\
\hline
problem & iter. of failure & $\min_k\{\|x_k-x^\ast\|\}$ & $k_{\max}=100$ & $k_{\max}=500$ & $k_{\max}=1000$ \\
\hline
HS7 & 77 & 7.8260E-3 & 1.0234E-3 & 4.9413E-8 & 4.9413E-8\\
BT11 & 64 & 4.8346E-2 & 3.9258E-3 & 1.9791E-6 & 1.4133E-6\\
HS40 & 26 & 3.4728E-2 & 2.1251E-3 & 1.09888E-6 & 1.0988E-6\\
\hline
\end{tabular} \label{relaxation1}
\end{table}

\begin{table}[h]
\caption{$\min_k\{\|x_k-x^\ast\|\}$ when $\epsilon_1=\epsilon_2=10^{-3}$}
\begin{tabular}{r | c | c | c | c | c }
\hline
& \multicolumn{2}{c}{relaxation disabled} & \multicolumn{3}{|c}{relaxation enabled} \\
\hline
problem & iter. of failure & $\min_k\{\|x_k-x^\ast\|\}$ & $k_{\max}=100$ & $k_{\max}=500$ & $k_{\max}=1000$ \\
\hline
HS7 & 42 & 8.0390E-2 & 1.0401E-3 & 4.9328E-6 & 4.9328E-6\\
BT11 & 18 & 9.3324E-1 & 4.0003E-3 & 1.9804E-4 & 1.4060E-4\\
HS40 & 6 & 6.4144E-2 & 2.2293E-3 & 1.1183E-4 & 4.9328E-6\\
\hline
\end{tabular}\label{relaxation2}
\end{table}

\begin{table}[h]
\caption{$\min_k\{\|x_k-x^\ast\|\}$ when $\epsilon_1=\epsilon_2=10^{-1}$}
\begin{tabular}{r | c | c | c | c | c }
\hline
& \multicolumn{2}{c}{relaxation disabled} & \multicolumn{3}{|c}{relaxation enabled} \\
\hline
problem & iter. of failure & $\min_k\{\|x_k-x^\ast\|\}$ & $k_{\max}=100$ & $k_{\max}=500$ & $k_{\max}=1000$ \\
\hline
HS7 & 10 & 3.7404E-1 & 1.3113E-3 & 4.5607E-4 & 2.5422E-4\\
BT11 & 8 & 1.7108 & 2.0598E-2 & 2.0598E-2 & 1.9451E-2 \\
HS40 & 2 & 1.1817E-1 & 5.8202E-2 & 3.8673E-2 &  3.8673E-2\\
\hline
\end{tabular}\label{relaxation3}
\end{table}

\paragraph{Effect of incorrect noise level estimations.}  In Algorithm 1, estimations of $\epsilon_f$ and $\epsilon_c$ are needed to set the relaxation bound $\epsilon_R$ in \eqref{relaxed}.  It is clear that underestimating the noise level can cause failure of the relaxed line search, which never fails when the true level (or an overestimation) is provided. On the other hand, overestimation 
can lead to large oscillations. The precise behavior of the algorithm will depend on the stop test, and there is no universally adopted stopping criterion in the noisy setting, to our knowledge.

Nevertheless, we performed the following experiments using a stop test that that could be considered as a naive modification of termination tests in standard packages.  We simply terminate the algorithm when the observed (noisy) feasibility and optimality errors are smaller than the (estimated) noise provided for these quantities, i.e.,
\begin{equation} \label{optterm}
\|\tilde c(x_k)\|_1 \leq \epsilon_c^{est} \quad \mbox{and} \quad \|\tilde g(x_k) + \tilde J(x_k)^T\lambda_k\| \leq \epsilon_g^{est} + \|\lambda_k\|_\infty\epsilon_J^{est}.
\end{equation}
Figures~\ref{estimation1}--\ref{estimation3} report the quantity $\min_k\{\|x_k-x^\ast\|\}$ when the algorithm employs estimated noise levels $\epsilon_1^{est}$ and $\epsilon_2^{est}$ that are 10, 100 and 1000 times larger or smaller than the correct values. We perform this experiment for $\epsilon_i = 10^{-1}, 10^{-3}, 10^{-5}$.  A termination due to the satisfaction of the  condition \eqref{optterm} is marked with \emph{(opt)}, and a line search failure is marked with \emph{(ls)}.

\begin{table}[H]
\caption{$\min_k\{\|x_k-x^\ast\|\}$ when true $\epsilon_i=10^{-5}; \  i=1,2$}
\begin{tabular}{r | c | c | c | c | c | c}
\hline
& \multicolumn{2}{c}{$\epsilon_i^{est} = \epsilon_i$} & \multicolumn{2}{|c}{$\epsilon_i^{est} = 0.001\epsilon_i$} & \multicolumn{2}{|c}{$\epsilon_i^{est} = 1000\epsilon_i$}\\
\hline
problem & iter.  & {\small$\min_k\{\|x_k-x^\ast\|\}$} & iter.  & {\small$\min_k\{\|x_k-x^\ast\|\}$} & iter. & {\small$\min_k\{\|x_k-x^\ast\|\}$} \\
\hline
HS7 & 188 (opt) & 3.2017E-6 &  69 (ls) & 8.8000E-3 &  74 (opt) & 5.5704E-3 \\
BT11 & 233 (opt) & 2.4010E-6 & 64 (ls) & 4.8346E-2 & 64 (opt) & 4.0112E-2 \\
HS40 & 2703 (opt) & 8.0766E-7 &  26 (ls) & 3.4728E-2 & 27 (opt) & 2.8305E-2 \\
\hline
\end{tabular}\label{estimation3}
\end{table}

\begin{table}[H]
\caption{$\min_k\{\|x_k-x^\ast\|\}$ when true $\epsilon_i=10^{-3}; \  i=1,2$}
\begin{tabular}{r | c | c | c | c | c | c}
\hline
& \multicolumn{2}{c}{$\epsilon_i^{est} = \epsilon_i$} & \multicolumn{2}{|c}{$\epsilon_i^{est} = 0.01\epsilon_i$} & \multicolumn{2}{|c}{$\epsilon_i^{est} = 100\epsilon_i$}\\
\hline
problem & iter.  & {\small$\min_k\{\|x_k-x^\ast\|\}$} & iter.  & {\small$\min_k\{\|x_k-x^\ast\|\}$} & iter. & {\small$\min_k\{\|x_k-x^\ast\|\}$} \\
\hline
HS7 & 117 (opt) & 3.5750E-4 & 42 (ls) & 8.0390E-2 & 39 (opt) & 5.4414E-2 \\
BT11 & 149 (opt) & 2.7466E-4 & 29 (ls) & 5.3925E-1 & 22 (opt) & 5.9597E-1 \\
HS40 & 154 (opt) & 4.2653E-4 & 7 (ls) & 6.4142E-2 & 2 (opt) & 6.9002E-2 \\
\hline
\end{tabular}\label{estimation2}
\end{table}

\begin{table}[H]
\caption{$\min_k\{\|x_k-x^\ast\|\}$ when true $\epsilon_i=10^{-1}; \  i=1,2$}
\begin{tabular}{r | c | c | c | c | c | c}
\hline
& \multicolumn{2}{c}{$\epsilon_i^{est} = \epsilon_i$} & \multicolumn{2}{|c}{$\epsilon_i^{est} = 0.1\epsilon_i$} & \multicolumn{2}{|c}{$\epsilon_i^{est} = 10\epsilon_i$}\\
\hline
problem & iter.  & {\small$\min_k\{\|x_k-x^\ast\|\}$} & iter.  & {\small$\min_k\{\|x_k-x^\ast\|\}$} & iter. & {\small$\min_k\{\|x_k-x^\ast\|\}$} \\
\hline
HS7 & 51 (opt) & 2.6752E-2 & 556 (ls) & 2.5682E-4 &  5 (opt) & 4.2796E-1 \\
BT11 & 20 (opt) & 6.6650E-1 & 3233 (ls) & 6.8738E-3 & 2 (opt) & 2.4045 \\
HS40 & 210 (opt) & 5.82E-2 &  982 (ls) & 1.4785E-2 &  0 (opt) & 2.8877E-1 \\
\hline
\end{tabular}\label{estimation1}
\end{table}

As expected, underestimations cause line search failures while overestimations cause \eqref{optterm} to be triggered at earlier iterations.  Another consequence of underestimating $\epsilon_2$ is that the algorithm might never be able to satisfy \eqref{optterm}, even if a line search failure occurs sufficiently late in the run; see for example the entry corresponding to   $\epsilon_i=10^{-1}, \epsilon_i^{est}=0.1\epsilon_i$. 
In summary, over-and underestimation of the noise levels can be harmful in ways that are dependent on the implementation.

We must point out that an optimization algorithm may provide an indication that the noise estimates must be re-computed. For example, the recovery procedure described by Berahas et al. \cite{berahas2019derivative} uses information from the line search to request a better estimate (e.g. through sampling or finite difference tables), and can take precautions to avoid harmful iterations. Robust implementations of methods for constrained optimization in the presence of noise should include such features.   




\section{Final Remarks}

Two questions guided this research. What is the best behavior one can expect of a constrained optimization method when functions and constraints contain a moderate amount of bounded noise that cannot be diminished at will? What are the minimal modifications of a classical optimization algorithm that allow it to tolerate noise, when the noise level can be estimated?

In this paper, we focused on a classical sequential quadratic programming method applied to equality constrained problems. We showed that a modification (relaxation) of the line search allows the iterates to approach a region around the solution where noise dominates---and that the iterates remain in a vicinity of this region, under normal circumstances. The analysis is presented under  benign assumptions, for example that the Jacobian of the constraints is never close to singular, which facilitates the choice of the penalty parameter. Nevertheless, we believe that the essence of the analysis captures some of the main challenges to be confronted when functions and derivatives contain noise. The accuracy bounds presented in this paper will be sharpened in a forthcoming paper that studies the local behavior of the method near a well behaved minimizer.

The thorny issue of how to design a proper stop test that reflects the desires of the users has not been addressed in this paper and is worthy of research. The treatment of singularity and the use of a nondiagonal Hessian also requires attention, as well as the very important question of how to handle noisy inequality constraints.

\bigskip\noindent
{\em Acknowledgement}. We thank Shigeng Sun for his careful reading of the paper and useful suggestions.


\newpage

\bibliographystyle{plain}

\bibliography{../References/references}
\end{document}